\documentclass[11pt]{article}

\usepackage[a4paper,margin=1.15in]{geometry}
\usepackage{amsmath,amssymb,amsthm}
\usepackage{mathtools}
\usepackage[colorlinks=true,linkcolor=blue,citecolor=blue,urlcolor=blue]{hyperref}
\hypersetup{
  pdftitle={Exact Minima of Finite Absolute Cosine Sums},
  pdfauthor={RuiXi Sun},
  pdfsubject={Exact minima and equality cases for finite absolute cosine sums}
}

\newtheorem{theorem}{Theorem}
\newtheorem{lemma}{Lemma}
\newtheorem{remark}{Remark}

\title{Exact Minima of Finite Absolute Cosine Sums}
\author{RuiXi Sun}
\date{}

\begin{document}

\maketitle

\begin{abstract}
We give a self-contained proof of the exact minimum
\(M_n=\min_{x\in\mathbb R}\sum_{k=1}^n |\cos(kx)|\) for every positive
integer \(n\). One has \(M_n=\lfloor n/2\rfloor\) when
\(n\notin\{2,4,6\}\), whereas
\[
M_2=\frac1{\sqrt2},\qquad
M_4=1+\frac{\sqrt3}{2},\qquad
M_6=
\frac{-1+3\sqrt5+2\sqrt{5+2\sqrt5}}4.
\]
For \(n\notin\{2,4,6\}\), the minimizers are exactly
\(x\equiv\pi/2\pmod\pi\). For \(n=2,4,6\), they are respectively
\(x\equiv\pm\pi/4\), \(\pm\pi/6\), and \(\pm\pi/10\pmod\pi\).
The elementary proof combines a reduction by piecewise concavity and
residue permutations with strict estimates for finite trigonometric sums.
\end{abstract}

\section{Introduction}

For a positive integer \(n\), consider
\[
        S_n(x)=\sum_{k=1}^n |\cos(kx)|.
\]
The point \(x=\pi/2\) gives
\[
        S_n\!\left(\frac{\pi}{2}\right)
        =\left\lfloor \frac n2\right\rfloor,
\]
so the natural question is whether this is the global minimum.

Special cases of this problem have earlier precedents. In his solution to
Problem M590, Gusyatnikov proves the two-term minimum \(M_2=1/\sqrt2\)
and states the sharp inequality \(S_4(x)\ge1+\sqrt3/2\) as an additional
exercise~\cite{Gusyatnikov1980}. Since equality in the latter occurs at
\(x=\pi/6\), the value of \(M_4\) is already contained in that problem
collection.

In 2022, HDD~292 posed the following inequality on
Mathematics Stack Exchange~\cite{MSE2022}:
\[
        |\cos x|+|\cos 2x|+\cdots+|\cos nx|
        \ge \left\lfloor\frac n2\right\rfloor,
        \qquad n\neq 2,4,6.
\]
The post contains more than the statement of the problem: it gives the
reduction to rational break points by piecewise concavity, the periodic
cotangent sum, and the lower bound obtained by selecting the smallest
distinct folded residues. In the notation used below, this is the bound
\(S_{q+r}(p\pi/(2q))\ge A(q,r)\). It also separates the ranges
\(n\ge2q\) and \(q\le n<2q\), but does not give a complete proof of
the estimates needed to finish the argument.

Related inequalities for termwise absolute values include the alternating
sums with triangular weights studied by Alzer, Liu and Shi~\cite{AlzerLiuShi2013} and the weighted absolute sine sums studied by
Alzer and Volkmer~\cite{AlzerVolkmer2026}. The alternating sine-square
and cosine-square inequalities of Alzer and Kwong~\cite{AlzerKwong2017} concern a different family. The position of the
absolute value also distinguishes the present problem from minima of
signed cosine sums, as in Odlyzko~\cite{Odlyzko1982}, and from the
integral \(L^1\) inequalities associated with Littlewood's problem,
established independently by McGehee, Pigno and Smith~\cite{McGeheePignoSmith1981} and Konyagin~\cite{Konyagin1982}.
Those integral inequalities concern the absolute value of an entire
exponential sum, rather than the pointwise sum of absolute values
considered here.

Starting from the reductions recorded in~\cite{MSE2022}, we supply
strict estimates in both ranges and give a unified proof for every
positive integer \(n\), including the exceptional values and all
equality cases. The contribution is the completion of these estimates
and the full minimizer classification; the question, the initial
reductions, and the values of \(M_2\) and \(M_4\) have the precedents
described above.

\section{Main result}

\begin{theorem}\label{thm:main}
For \(n\ge 1\), define
\[
        M_n=\min_{x\in\mathbb R} S_n(x),
        \qquad
        S_n(x)=\sum_{k=1}^n |\cos(kx)|.
\]
Then
\[
        M_n=\left\lfloor\frac n2\right\rfloor
        \qquad (n\notin\{2,4,6\}),
\]
and the exceptional values are
\[
        M_2=\frac1{\sqrt2},\qquad
        M_4=1+\frac{\sqrt3}{2},
\]
\[
        M_6=
        \frac{-1+3\sqrt5+2\sqrt{5+2\sqrt5}}4.
\]
Moreover, if \(n\notin\{2,4,6\}\), then
\[
        S_n(x)=M_n
        \quad\Longleftrightarrow\quad
        x\equiv \frac{\pi}{2}\pmod{\pi}.
\]
For the three exceptional cases the minimizers are, respectively,
\[
        x\equiv \pm\frac{\pi}{4},\qquad
        x\equiv \pm\frac{\pi}{6},\qquad
        x\equiv \pm\frac{\pi}{10}
        \pmod{\pi}.
\]
\end{theorem}

\section{Reduction to rational break points}\label{sec:reduction}

The reductions in this section and Section~\ref{sec:permutation} follow
the approach in~\cite{MSE2022}. We include their proofs to keep the
argument self-contained.

\begin{lemma}\label{lem:concavity}
Let \(n\ge 1\). Every global minimizer \(x_0\) of \(S_n\) is of the form
\[
        x_0=\frac{p\pi}{2q},
        \qquad
        1\le q\le n,
        \qquad
        (p,2q)=1.
\]
Here \(q\) may be chosen as the least positive integer for which \(\cos(qx_0)=0\).
\end{lemma}

\begin{proof}
Since \(S_n\) is \(\pi\)-periodic, it is enough to consider \(x\in[0,\pi]\). The zeros of the functions \(\cos(kx)\), \(1\le k\le n\), divide this interval into finitely many subintervals. On any open subinterval where no \(\cos(kx)\) vanishes, all signs are fixed, so
\[
        S_n(x)=\sum_{k=1}^n \varepsilon_k\cos(kx),
        \qquad \varepsilon_k\in\{-1,1\}.
\]
On that subinterval,
\[
        S_n''(x)
        =-\sum_{k=1}^n k^2|\cos(kx)|<0.
\]
Thus \(S_n\) is strictly concave on each such open interval, and a minimum on the corresponding closed interval occurs at an endpoint.

The endpoints \(0\) and \(\pi\) give \(S_n(0)=S_n(\pi)=n\), while
\[
        S_n\!\left(\frac{\pi}{2}\right)
        =\left\lfloor\frac n2\right\rfloor<n
\]
for \(n\ge2\). The case \(n=1\) is immediate. Hence, for \(n\ge2\), a global minimizer is a point at which \(\cos(qx_0)=0\) for at least one \(q\in\{1,\dots,n\}\).

Choose the least such \(q\). Then \(qx_0=p\pi/2\) for an odd integer \(p\), so
\[
        x_0=\frac{p\pi}{2q}.
\]
If \(d=(p,q)>1\), then \(d\) is odd and
\[
        \frac qd x_0=\frac{p/d}{2}\pi
\]
is an odd multiple of \(\pi/2\), contradicting the minimality of \(q\). Therefore \((p,q)=1\). Since \(p\) is odd, \((p,2q)=1\).
\end{proof}

The case \(q=1\) in Lemma~\ref{lem:concavity} is exactly
\[
        x_0\equiv \frac{\pi}{2}\pmod{\pi},
\]
and then \(S_n(x_0)=\lfloor n/2\rfloor\). It remains to show that all candidates with \(q\ge2\) are larger, except for the stated exceptional cases.

\section{\texorpdfstring{The permutation modulo \(2q\)}{The permutation modulo 2q}}\label{sec:permutation}

Throughout this section let
\[
        x_0=\frac{p\pi}{2q},\qquad (p,2q)=1,\qquad q\ge2,
\]
and put
\[
        \theta=\frac{\pi}{2q},
        \qquad
        a=\frac{\theta}{2}=\frac{\pi}{4q}.
\]

\begin{lemma}\label{lem:permutation}
As multisets,
\[
        \bigl\{|\cos(kx_0)|:1\le k\le q-1\bigr\}
        =
        \bigl\{\cos(k\theta):1\le k\le q-1\bigr\}.
\]
Consequently
\[
        \sum_{k=1}^{q}|\cos(kx_0)|
        =
        \sum_{k=1}^{q-1}\cos(k\theta)
        =
        \frac12\left(\cot a-1\right),
\]
and a full period satisfies
\[
        \sum_{k=1}^{2q}|\cos(kx_0)|=\cot a.
\]
\end{lemma}

\begin{proof}
Multiplication by \(p\) permutes the residue classes modulo \(2q\). If \(1\le k\le q-1\), then \(pk\not\equiv 0,q\pmod{2q}\). Fold the residue \(pk\) into \(\{1,\dots,q-1\}\) using the symmetry
\[
        |\cos(m\theta)|=|\cos((2q-m)\theta)|.
\]
This gives a map from \(\{1,\dots,q-1\}\) into itself. If two indices \(k,\ell\) have the same folded residue, then
\[
        pk\equiv \pm p\ell \pmod{2q}.
\]
Since \((p,2q)=1\), this gives \(k\equiv\pm \ell\pmod{2q}\). For \(1\le k,\ell<q\), the negative case would imply \(k+\ell=2q\), impossible. Hence \(k=\ell\), so the folded residues form a permutation.

The formula for the first \(q\) terms follows because \(\cos(qx_0)=0\). The trigonometric identity
\[
        \sum_{k=1}^{q-1}\cos(k\theta)
        =\frac12\left(\cot\frac{\theta}{2}-1\right)
\]
then gives the displayed value. Finally, over a complete period modulo \(2q\), the residue \(0\) contributes \(1\), the residue \(q\) contributes \(0\), and all other folded residues occur twice. Therefore
\[
        \sum_{k=1}^{2q}|\cos(kx_0)|
        =
        1+2\sum_{k=1}^{q-1}\cos(k\theta)
        =\cot a.
\]
\end{proof}

For \(0\le r<q\), define
\[
        A(q,r)
        =
        \sum_{j=1}^{q-1}\cos(j\theta)
        +
        \sum_{j=1}^{r}\sin(j\theta).
\]
Using the preceding notation,
\[
        A(q,r)
        =
        \frac{2\cos a-\cos((2r+1)a)}{2\sin a}
        -\frac12.
        \tag{1}\label{eq:Aformula}
\]

\begin{lemma}\label{lem:short-lower}
If \(n=q+r\) with \(0\le r<q\), then
\[
        S_n(x_0)\ge A(q,r).
\]
\end{lemma}

\begin{proof}
The first \(q\) terms are given by Lemma~\ref{lem:permutation}. For \(1\le j\le r\),
\[
        |\cos((q+j)x_0)|
        =
        |\sin(jx_0)|
        =
        |\sin(jp\theta)|.
\]
The same folding argument as in Lemma~\ref{lem:permutation} shows that these are \(r\) distinct elements of
\[
        \{\sin\theta,\sin2\theta,\dots,\sin((q-1)\theta)\}.
\]
Since these numbers are strictly increasing, the sum of any \(r\) of them is at least the sum of the \(r\) smallest ones. Hence
\[
        \sum_{j=1}^{r}|\cos((q+j)x_0)|
        \ge
        \sum_{j=1}^{r}\sin(j\theta),
\]
which proves the lemma.
\end{proof}

\section{\texorpdfstring{The long range \(n\ge 2q\)}{The long range n >= 2q}}

\begin{lemma}\label{lem:long}
Let \(q\ge2\), \(x_0=p\pi/(2q)\), and \((p,2q)=1\). If \(n\ge2q\), then
\[
        S_n(x_0)>\frac n2.
\]
\end{lemma}

\begin{proof}
Write
\[
        n=2q\ell+s,\qquad 0\le s<2q,\qquad \ell\ge1.
\]
By periodicity modulo \(2q\),
\[
        S_n(x_0)=(\ell-1)\sum_{k=1}^{2q}|\cos(kx_0)|
        +
        \left(
        \sum_{k=1}^{2q}|\cos(kx_0)|
        +
        \sum_{k=1}^{s}|\cos(kx_0)|
        \right).
\]
It is enough to prove, for \(0\le s<2q\),
\[
        \cot a+\sum_{k=1}^{s}|\cos(kx_0)|
        > q+\frac s2.
        \tag{2}\label{eq:long-target}
\]
Indeed, \(\cot a>q\) follows from the same proof below with \(s=0\), and then \eqref{eq:long-target} gives \(S_n(x_0)>n/2\).

First suppose \(0\le s<q\). By the permutation argument, the sum of the first \(s\) terms is at least the sum of the \(s\) smallest elements among
\[
        \{\cos\theta,\dots,\cos((q-1)\theta)\},
\]
namely
\[
        \sum_{j=1}^{s}\sin(j\theta).
\]
Set \(y=(2s+1)a\). Then
\[
        \cot a+\sum_{j=1}^{s}\sin(2ja)
        =
        \frac{3\cos a-\cos y}{2\sin a}.
\]
Since
\[
        q+\frac s2=\frac{\pi+y-a}{4a},
\]
it is enough, using \(\sin a<a\) and \(\cos a>1-a^2/2\), to prove
\[
        3-\cos y-\frac y2
        >
        \frac{\pi}{2}-\frac a2+\frac32a^2.
        \tag{3}\label{eq:long-first-case}
\]
On \(0\le y\le\pi/2\), the function \(3-\cos y-y/2\) is minimized at \(y=\pi/6\). Also \(0<a\le\pi/8\), and the right-hand side of \eqref{eq:long-first-case} is bounded above by its larger endpoint value on this interval. Hence it is enough to check
\[
        3-\frac{\sqrt3}{2}
        -\frac{25\pi}{48}
        -\frac{3\pi^2}{128}>0.
\]
Using \(\sqrt3<7/4\) and \(\pi<22/7\), the left-hand side is greater than
\[
        3-\frac78-\frac{275}{168}-\frac{363}{1568}
        =
        \frac{1207}{4704}>0.
\]

It remains to treat \(q\le s<2q\). Write \(s=q+t\), \(0\le t<q\), and set \(y=(2t+1)a\). The first \(q\) terms contribute \((\cot a-1)/2\), and the remaining \(t\) terms are bounded below by \(\sum_{j=1}^t\sin(j\theta)\). Thus
\[
        \cot a+\sum_{k=1}^{s}|\cos(kx_0)|
        \ge
        \frac{4\cos a-\cos y}{2\sin a}-\frac12.
\]
Arguing as before, it is enough to prove
\[
        4-\cos y-\frac y2
        >
        \frac{3\pi}{4}+\frac a2+2a^2.
\]
The left-hand side is minimized at \(y=\pi/6\), and \(a\le\pi/8\). Hence it is enough to verify
\[
        4-\frac{\sqrt3}{2}
        -\frac{43\pi}{48}
        -\frac{\pi^2}{32}>0.
\]
Again using \(\sqrt3<7/4\) and \(\pi<22/7\), the left-hand side is greater than
\[
        4-\frac78-\frac{473}{168}-\frac{121}{392}
        =
        \frac1{1176}>0.
\]
This proves \eqref{eq:long-target} in both cases.
\end{proof}

\section{\texorpdfstring{The short range \(q\le n<2q\)}{The short range q <= n < 2q}}

\begin{lemma}\label{lem:short}
Let \(q\ge2\) and \(0\le r<q\). If \((q,r)\notin\{(2,0),(3,1),(5,1)\}\), then
\[
        A(q,r)>\left\lfloor\frac{q+r}{2}\right\rfloor .
\]
For the three excluded pairs,
\[
        A(2,0)=\frac1{\sqrt2},\qquad
        A(3,1)=1+\frac{\sqrt3}{2},
\]
and
\[
        A(5,1)=
        \frac{-1+3\sqrt5+2\sqrt{5+2\sqrt5}}4.
\]
\end{lemma}

\begin{proof}
First assume \(q\ge6\). Let \(a=\pi/(4q)\) and \(y=(2r+1)a\). By \eqref{eq:Aformula}, the inequality
\[
        A(q,r)>\frac{q+r}{2}
\]
is equivalent to
\[
        2\cos a-\cos y>(q+r+1)\sin a.
\]
Since
\[
        q=\frac{\pi}{4a},
        \qquad
        r=\frac{y/a-1}{2},
\]
we have
\[
        q+r+1=\frac{\pi+2y+2a}{4a}.
\]
Using \(\sin a<a\) and \(\cos a>1-a^2/2\), it is enough to prove
\[
        2-\cos y-\frac y2
        >
        \frac{\pi}{4}+\frac a2+a^2.
        \tag{4}\label{eq:short-core}
\]
The function \(2-\cos y-y/2\) is minimized on \(0\le y\le\pi/2\) at \(y=\pi/6\). Since \(q\ge6\), \(a\le\pi/24\). Thus \eqref{eq:short-core} follows from
\[
        2-\frac{\sqrt3}{2}
        -\frac{17\pi}{48}
        -\frac{\pi^2}{576}>0.
\]
With \(\sqrt3<26/15\) and \(\pi<22/7\), the left-hand side is greater than
\[
        2-\frac{13}{15}-\frac{187}{168}-\frac{121}{7056}
        =
        \frac{109}{35280}>0.
\]
Hence \(A(q,r)>(q+r)/2\), and in particular it is larger than \(\lfloor(q+r)/2\rfloor\).

It remains to check \(2\le q\le5\). Direct evaluation gives
\[
\begin{array}{c|c|c}
(q,r) & A(q,r) & \text{comparison}\\
\hline
(2,0) & 1/\sqrt2 & <1\\
(2,1) & \sqrt2 & >1\\
(3,0) & (1+\sqrt3)/2 & >1\\
(3,1) & 1+\sqrt3/2 & <2\\
(3,2) & 1+\sqrt3 & >2\\
(4,0) & \sum_{j=1}^{3}\cos(j\pi/8) & >2\\
(4,1) & A(4,0)+\sin(\pi/8) & >2\\
(4,2) & A(4,1)+\sin(\pi/4) & >3\\
(4,3) & 2A(4,0) & >3\\
(5,0) & \sum_{j=1}^{4}\cos(j\pi/10) & >2\\
(5,1) & C_5+\sin(\pi/10) & <3\\
(5,2) & A(5,1)+\sin(\pi/5) & >3\\
(5,3) & A(5,2)+\sin(3\pi/10) & >4\\
(5,4) & 2C_5 & >4
\end{array}
\]
where \(C_5=\sum_{j=1}^{4}\cos(j\pi/10)\). For completeness, the
comparisons for \(q=4,5\) can be checked using the following exact
bounds. Write \(s=\sqrt2\). Then
\[
        A(4,0)=\frac{\sqrt{4+2s}+s}{2}>2,
        \qquad
        \sin\frac\pi8>\frac38,
        \qquad
        \sin\frac\pi4>\frac7{10}.
\]
The first inequality follows by squaring from \(s>7/5\); the others
follow from \(7/5<s<23/16\). Thus
\(A(4,2)>2+3/8+7/10>3\), and the other comparisons for \(q=4\)
follow by positivity.

Next write \(t=\sqrt5\) and \(u=\sqrt{5+2t}\). The exact values give
\[
        C_5=\frac{t+u}{2}>\frac52,\qquad
        \sin\frac\pi{10}>\frac14,\qquad
        \sin\frac\pi5>\frac12,\qquad
        \sin\frac{3\pi}{10}>\frac34.
\]
These bounds imply all the lower comparisons for \(q=5\). Also,
\(t<9/4\) and \(u<25/8\) give
\(A(5,1)=(-1+3t+2u)/4<3\). In particular,
\[
        C_5+\sin(\pi/10)
        =
        \cos\frac{\pi}{10}
        +\cos\frac{\pi}{5}
        +\cos\frac{3\pi}{10}
        +2\cos\frac{2\pi}{5}
\]
equals
\[
        \frac{-1+3\sqrt5+2\sqrt{5+2\sqrt5}}4.
\]
This proves the lemma.
\end{proof}

\section{Proof of the theorem}

\begin{proof}[Proof of Theorem~\ref{thm:main}]
The function \(S_n\) is continuous and \(\pi\)-periodic, so a global
minimizer exists. Let \(x_0\) be such a minimizer and choose \(q\) as
in Lemma~\ref{lem:concavity}.

If \(q=1\), then
\[
        x_0\equiv \frac{\pi}{2}\pmod{\pi}
\]
and
\[
        S_n(x_0)=\left\lfloor\frac n2\right\rfloor.
\]

Now assume \(q\ge2\). If \(n\ge2q\), Lemma~\ref{lem:long} gives
\[
        S_n(x_0)>\frac n2\ge \left\lfloor\frac n2\right\rfloor.
\]
Since \(S_n(\pi/2)=\lfloor n/2\rfloor\), no global minimizer can
occur in the long range.

It remains to consider the short range \(q\le n<2q\). Write
\[
        n=q+r,\qquad 0\le r<q.
\]
By Lemma~\ref{lem:short-lower},
\[
        S_n(x_0)\ge A(q,r).
\]
By Lemma~\ref{lem:short}, this is strictly larger than \(\lfloor n/2\rfloor\), except possibly for
\[
        (q,r)=(2,0),(3,1),(5,1),
\]
which correspond exactly to
\[
        n=2,\qquad n=4,\qquad n=6.
\]
Thus, for every \(n\notin\{2,4,6\}\), the only minimizers are the points with \(q=1\), namely
\[
        x\equiv\frac{\pi}{2}\pmod{\pi},
\]
and
\[
        M_n=\left\lfloor\frac n2\right\rfloor.
\]

It remains only to record the exceptional cases. For \(n=2\), the candidates from Lemma~\ref{lem:concavity} have \(q=1\) or \(q=2\). The value at \(q=1\) is \(1\), while
\[
        S_2\!\left(\frac{\pi}{4}\right)
        =
        \frac1{\sqrt2}.
\]
Hence
\[
        M_2=\frac1{\sqrt2},
        \qquad
        x\equiv\pm\frac{\pi}{4}\pmod{\pi}.
\]

For \(n=4\), the only possible candidates with value below \(2\)
come from the short-range pair \((q,r)=(3,1)\). Here
\((p,6)=1\) means \(p\equiv\pm1\pmod6\), and direct substitution
gives \(S_4(\pi/6)=A(3,1)\). Consequently,
\[
        M_4=A(3,1)=1+\frac{\sqrt3}{2},
        \qquad
        x\equiv\pm\frac{\pi}{6}\pmod{\pi}.
\]
All candidates with \(q=1\) give value \(2\), candidates with \(q=2\) are in the long range and have value \(>2\), and candidates with \(q=4\) have value \(>2\) by Lemma~\ref{lem:short}.

For \(n=6\), the only possible candidates with value below \(3\)
come from \((q,r)=(5,1)\). For these candidates,
put \(C_5=\sum_{j=1}^4\cos(j\pi/10)\).
Lemma~\ref{lem:permutation} gives the exact identity
\[
        S_6\!\left(\frac{p\pi}{10}\right)
        =C_5+\left|\sin\frac{p\pi}{10}\right|
        \ge C_5+\sin\frac\pi{10}=A(5,1).
\]
Equality holds at \(p=1\), so \(S_6(\pi/10)=A(5,1)\) and
\[
        M_6=A(5,1)
        =
        \cos\frac{\pi}{10}
        +\cos\frac{\pi}{5}
        +\cos\frac{3\pi}{10}
        +2\cos\frac{2\pi}{5},
\]
that is,
\[
        M_6=
        \frac{-1+3\sqrt5+2\sqrt{5+2\sqrt5}}4.
\]
Since
\[
        \left|\sin\frac{p\pi}{10}\right|=\sin\frac\pi{10}
        \quad\Longleftrightarrow\quad
        p\equiv\pm1\pmod{10},
\]
the minimum is attained exactly at
\[
        x\equiv\pm\frac{\pi}{10}\pmod{\pi}.
\]
\end{proof}

\begin{remark}
For \(n\notin\{2,4,6\}\), the strict estimates exclude every rational
break point with least denominator \(q\ge2\). The equality
classification therefore follows from these estimates, without a
separate comparison of the numerators \(p\).
\end{remark}

\begingroup
\small
\raggedright

\endgroup

\end{document}